\documentclass[reqno,12pt,a4paper]{amsart}
\usepackage{amsmath,amssymb,amsthm,amscd,amssymb,amsfonts,hyperref}
\usepackage[english]{babel}
\usepackage[mathscr]{eucal}

\numberwithin{equation}{section}

\renewcommand{\a}{\alpha}
\renewcommand{\b}{\beta}

\newcommand{\G}{\Gamma}
\renewcommand{\d}{\delta}
\newcommand{\D}{\Delta}
\newcommand{\e}{\epsilon}
\newcommand{\z}{\zeta}

\renewcommand{\k}{\kappa}
\newcommand{\vk}{\varkappa}
\renewcommand{\l}{\lambda}

\newcommand{\m}{\mu}
\newcommand{\n}{\nu}

\newcommand{\X}{\Xi}
\renewcommand{\r}{\rho}
\newcommand{\s}{\sigma}

\renewcommand{\t}{\tau}
\newcommand{\f}{\phi}
\newcommand{\F}{\Phi}
\newcommand{\h}{\chi}
\newcommand{\p}{\psi}

\renewcommand{\o}{\omega}
\renewcommand{\O}{\Omega}

\newcommand{\C}{{\mathbb C}}
\newcommand{\R}{{\mathbb R}}

\newcommand{\Fc}{{\mathcal F}}

\newcommand{\Qc}{{\mathcal Q}}

\DeclareMathOperator{\im}{{\rm Im}\,}
\DeclareMathOperator{\re}{{\rm Re}\,}

\newtheorem{theorem}{Theorem}[section]
\newtheorem{proposition}[theorem]{Proposition}

\theoremstyle{definition}

\theoremstyle{remark}

\newtheorem{example}[theorem]{Example}

\begin{document}

\title[zero modes]{Entire functions in weighted $L_2$ and   zero modes of the Pauli operator with non-signdefinite magnetic field}
\author[Rozenblum]{Grigori Rozenblum}
\address[Rozenblum]{Department of Mathematics
                        Chalmers University of Technology \\
                        and Department of Mathematics University of Gothenburg \\
                         S-412 96 Gothenburg \\
                        Sweden}
\email{grigori@math.chalmers.se}
\author[Shirokov]{Nikolay Shirokov}
\address[Shirokov]{Department of Mathematics and Mechanics \\
                      St. Petersburg State University\\
                      Russia}
\email{nikolai.shirokov@gmail.com}

\begin{abstract} For a real non-signdefinite  function $B(z)$, $z\in \C$, we investigate
the dimension of the space of entire analytical
functions square integrable with weight $e^{\pm 2F}$,
where the function $F(z)=F(x_1,x_2)$ satisfies the
Poisson equation $\D F=B$. The answer is known for the
function $B$ with constant sign. We discuss some
classes of non-signdefinite positively homogeneous
functions $B$, where both infinite and zero dimension
may occur. In the former case we present a method of
constructing entire functions with prescribed behavior
at infinity in different directions. The topic is
closely related with the question of the dimension of
the zero energy subspace (zero modes) for the Pauli
operator.
\end{abstract}

\date{\today}
 \maketitle


\section{Introduction}
\label{intro} In 1979 Y.Aharonov and A.Cacher in
\cite{ahacash79} discovered that the Pauli operator in
dimension 2 with a   compactly supported bounded
magnetic field $B(x),\; x=(x_1,x_2)$, can possess zero
modes, eigenfunctions with zero energy. The number of
these zero modes (the dimension of the zero energy
eigenspace) is finite and is determined by the total
flux of the magnetic field. The  zero modes problem
has been investigated further on and  the
Aharonov-Casher formula was extended to rather
singular and not compactly supported magnetic fields
being signed measures with finite total variation
(\cite{erdosvug02}). On the other hand, sign-definite
fields with infinite flux, the authors proved in
\cite{rozshir} that the space of zero modes is
infinite-dimensional, thus extending the
Aharonov-Casher formula to this case. Moreover, the
infiniteness of zero modes was established in
\cite{rozshir} for a class of magnetic fields with
variable sign, such that in certain sense the part
having one direction is infinitesimal with respect to
the part with another direction, while both parts have
infinite flux, as well for weakly perturbed constant
magnetic fields. On the other hand, in
\cite{erdosvug02} an example was constructed of a
magnetic field consisting of tiny islands, sparsely
placed in the plane, carrying positive magnetic field,
on the background of annuli with negative field, such
that both positive and negative parts of the field
have infinite total flux, so that no zero modes exist.
So, it was, generally, unclear, what is the situation
with zero modes for the case when neither of
sign-parts of the magnetic field prevails over the
other one. After having been acquainted with
\cite{rozshir}, B.Simon asked the first author (G.R.)
about the number of zero modes for a very simple
configuration of the field of this kind: some constant
with one sign in one half-plane and a constant with
different sign in another one. The answer (no zero
modes at all) was found quite easily, but a more
general question arose: how many zero modes are
generated by the magnetic field which is constant in a
sector in the plane and constant, with different sign,
in the complement of this sector, or, more generally,
by a non-signdefinite radial-homogeneous field. The
present paper contains some results in this direction.
For the sector case, it turns out that if one of the
sectors is sufficiently small, the space of zero modes
is infinite-dimensional. On the other hand, if the
angles of the sectors  are sufficiently close to
$\pi$, zero modes are completely absent. Somewhat
similar situation takes place for fields with some
other degree of homogeneity.

Starting from the paper \cite{ahacash79}, it became
clear that the progress in the  zero modes problem
depends heavily on the properties of solutions of the
Poisson equation $\Delta F(x)=B(x)$ in the terms of
$B(x)$. The zero modes are entire analytical (or
anti-analytical) functions $u(z)$ of the variable
$z=x_1+ix_2$ such that $u\exp(\pm F)\in L_2(\R^2)$. If
$B(x)=B>0$ is a nonzero constant the equation has a
solution of the form $F(x)=\frac{B}{2}|x|^2$, and this
fact obviously leads to the infiniteness of the
dimension of the space of analytical functions with
$u\exp(- F)\in L_2(\R^2)$. However, generally, the
boundedness of $B$ does not guarantee by itself a
quadratic estimate for $F$, moreover, it may happen
that the Poisson equation has no semi-bounded
solutions, and such a straightforward reasoning about
zero modes fails. We need a deeper analysis of entire
functions, square integrable with weight $\exp(\pm
2F)$, without the condition imposed, that $F$ is
subharmonic (the subharmonic case is investigated
exhaustively in \cite{horm} and \cite{shigekawa}).

We start in Section 2 by considering solutions of the
Poisson equations for a radial homogeneous right-hand
side. In particular,  for our 'sector' configuration
of the field we  construct a special solution of the
Poisson equation. This solution $F$ is not
semi-bounded, behaves itself at infinity as
$C|x|^2\log|x|$ in all directions but four, with $C$
depending on the direction and having variable sign.
Next, in Sect.3, we construct entire analytical
functions $u$ such that $u\exp (-F)\in L_2$. Such
functions $u$ must decay rather rapidly in directions
where $F$ is negative, and they may grow,  but in a
controllable way, in directions where $F$ is positive.
We present a method for constructing entire functions
with such behavior. This construction can be put
through provided the angle of the sector where $B$ is
negative is sufficiently small. So, it turns out that
in this latter case there are infinitely many zero
modes.

We consider the case of a sector with an astute angle
in Sect.4. We show that in this case there are no zero
modes at all, provided that the angle of the sector is
sufficiently close to $\pi$. The reason for this is
that, if an entire function  is square integrable in a
sector with sufficiently fast growing weight, it must
be zero everywhere, disregarding its behavior outside
the sector. This idea requires certain work for being
implemented in our case, since for an sector type
magnetic field the sets where the potential has fixed
sign differ only slightly from the quarter-plane so
that logarithmic effects must be taken into account.

In the final section we briefly consider magnetic
fields, radial  homogeneous of some negative order.

The main part of the results were obtained when the
second author (N.Sh.) was enjoying the hospitality of
Chalmers University of Technology, being  supported by
a grant from the Swedish Royal Academy of Sciences
(KVA).
\section{General constructions}\label{general}
\subsection{Homogeneous solutions of the Poisson equation}
We identify the real plane $\R^2$ with the complex
plane $\C^1$, setting $z=x_1+ix_2$; by $d\m$ we denote
the Lebesgue measure on the plane. Let $B(x)$ be a
real-valued function in $\R^2$ positively homogeneous
of degree $s$. Then, as it is well known,  the
solution of the Poisson equation
\begin{equation}\label{Poisson}
    \D F(x)=B(x)
\end{equation}
 can be
looked for as a positively homogeneous function of
degree $s+2$. In fact, if the function $B$ has the
form $b(\p) r^s$ in polar coordinates $(r,\p)$, we can
look for $F$ in the form $\varphi(\p)r^{s+2}$ and
obtain an equation for $\varphi$

\begin{equation}\label{2:EqForF}
    \varphi(\p)''+(s+2)^2 \varphi(\p)=b(\p)
\end{equation}

If $s$ in not an integer,  \eqref{2:EqForF} has a
unique solution, and thus $F=\varphi(\p)r^{s+2}$ is a
solution for \eqref{Poisson}. However, if $s$ is an
integer, the equation \eqref{2:EqForF} is solvable
only for those $b$ which are orthogonal to
$\exp(i(s+2)\p)$. If this orthogonality condition is
not satisfied, the solution of \eqref{Poisson} must
contain a logarithmic factor,
\begin{equation}\label{2:log}
    F(r,\p)=A\sin((s+2)(\p-\p_0))r^{s+2}\log r
    +\varphi(\p)r^{s+2},
\end{equation}
with properly selected $A$ and $\p_0$. Such case will
be referred to as the resonance one; if the
orthogonality condition is met as well as for a
noninteger $s$ we have the non-resonance case.

 For $0<\a<\pi$ we
denote by $\O_1$ the sector $\p=\arg z\in(0,\a)$ and
by $\O_2$ the complementing sector in the  plane.
Having two numbers, $b_1<0<b_2$, we set $B(x)\equiv
B(x)=2b_1$ for $x\in \O_1$ and $B(x)=2b_2$ in $\O_2$.
By scaling, one can reduce the situation to the case
$b_2=1$ and we always suppose that it is already done.

We are looking for a solution $F(x)$ of the equation
\eqref{Poisson}. Since the homogeneity degree equals
$s=0$, the solution must contain a power-logarithmical
term as in \eqref{2:log}. It is convenient to write
the solution of \eqref{Poisson} in a somewhat
different form.

\subsection{A solution of the Poisson equation for the sector configuration}

 We are looking for an
explicit formula for $F$. This function will be
constructed step-wise. We start by  elementary
solutions separately in $\O_1$ and $\O_2$. These
solutions do not fit together on the ray $\arg z=\a$.
Then some correction terms will be introduced.

So, we start with
\begin{equation*}
    \f(z)= b_1 x_2^2,  z\in\O_1; \f(z)=
          x_2^2,  z\in\O_2.
\end{equation*}
This function satisfies the equation \eqref{Poisson}
everywhere except the ray $L_\a=\{\arg z=\a\}$ where
it is discontinuous. To compensate this jump, as well
as the jump of its derivative, we will use the branch
of logarithm, continuous in the domain
$\C^{(\a)}=\C\setminus L_\a$. For the function
$\xi(z)=\frac{(ze^{-i\a})^2}{2\pi }\log z$ the
imaginary part  has a jump on $L_\a$ while the real
part is continuous but has a discontinuous derivative.
Our first correction  will make the whole solution
continuous on $L_\a$. The jump of $\f$ at the point
$z_0=r_0e^{i\a}$ equals
\begin{equation*}
    \f(z_0)|_{\O_1}-\f(z_0)|_{\O_2}=-c_0r_0^2\sin^2\a,
\end{equation*}
$c_0=1-b_1$, therefore we set
\begin{equation}\label{2:Corr1}
    \f_1(z)=\f(z)+c_0\sin^2\a\im (\xi(z)).
\end{equation}
Since the jump of the second summand in
\eqref{2:Corr1}  at the point $z_0=r_0 e^{i\a}$ equals
$c_0 r_0^2\sin^2\a$, the function $\f_1$ is continuous
in $\C$.

We consider now the  the derivatives of $\f_1$ at
$L_\a$. Obviously, the derivative along the ray is
continuous. The derivative across the ray has a jump,
and we will compensate this jump by subtracting the
real part of $\xi(z)$ with a proper coefficient. We
set
    $F(z)=\f_1(z)-c_0\sin\a\cos\a\re(\xi(z)).$
Since we added the real and imaginary parts of
functions that are  analytic in $\C^{(\a)}$, the
Poisson equation \eqref{Poisson} will be satisfied by
$F$ in $\C^{(\a)}$. The function $F$ and its
derivatives are continuous everywhere, thus the
distributional Laplacian  of $F$ coincides with the
classical Laplacian, and therefore $F$ is the solution
we need. To get a better understanding of $F$, we
represent it in a little bit different way:
\begin{gather}\label{2:Final}
    F(z)=\f(z)+\frac{c_0\sin^2\a}{2\pi}
    \re\left( \frac1{i}(ze^{i\a})^2\log z\right)-\\\nonumber\frac{c_0\sin\a\cos\a}{2\pi}
    \re\left((ze^{i\a})^2\log z\right)
    =\f(z)-\frac{c_0\sin\a}{2\pi}\re\left(\left(ze^{-\frac{i\a}{2}}\right)^2\log z\right).
\end{gather}
 Note that the behavior of $F(z)$ for large $|z|$ is determined by the second, power-logarithmic term in
\eqref{2:Final},
 except the directions where it vanishes, i.e., except the directions
  $\arg(z)=\frac\a{2}+k\frac{\pi}4$, $k=0,1,2,3.$
 These half-lines divide $\C$ in four quarters, in two of those the
  function $F$ grows as $C|z|^2\log|z|$, with
 some positive $C$ (depending on the direction), in the other two this
 functions tends to $-\infty$, again
  like $C|z|^2\log|z|$ but with a negative $C$ this time.

 In the next Section we will construct entire analytical functions $u(z)$
  such that $u\exp(F)\in L_2.$
  \section{Existence of zero modes }\label{existence}
  The aim of this Section is to establish the
  following fact concerning the sector
  configuration, as in Subsection 2.2.
  \begin{theorem}\label{thm:exist} Suppose that the size $\a$ of the sector
  and $b_1$ are sufficiently small. Then the space of
  entire
  analytical functions $u(z)$ satisfying $u\exp(F)\in
  L_2$ is infinite-dimensional.
 \end{theorem}
\subsection{Construction of a subharmonic function}\label{3}
 In this
subsection we  construct a subharmonic function of a
special form, to be used further on in the
construction of analytical functions with prescribed
behavior at infinity.
 We fix some positive $\e$, to be determined later.
 Consider two sectors $\Theta_j^\circ=\{z:|\arg z-\pi j|<\e\}$, $j=0,1,$ and
  set $\Theta_j=\Theta_j^\circ\cap\{|z|>1\}$.
  For some fixed $\s$, we  cut each of the sectors into strips by straight lines $\im z= k\s;\ k=0,\pm1,\pm2,\dots$.
  Starting from the  boundary
  lying closest to the imaginary axis, we cut each such strip by lines
   parallel to the imaginary axis,
  into domains having area $\s^2$. Just a finite number  of such domains are not polygons, a few of domains in each strip
    are triangles or trapezia, all the rest are unit squares.
   We will denote generically
  all these pieces of different form by $Q$
   and the set of these domains by $\Qc$; by $\Qc_j$ we denote the set of pieces in $\Theta_j$.
  For each  $Q\in\Qc$ we select a point $a_Q\in Q$ in the following way.
If $Q$ is a square we take the center of $Q$ as $a_Q$.
Otherwise we choose $a_Q$ so that
$\int_Q(z-a_Q)d\m=0.$ A simple geometrical
consideration shows that the distance between such
points is not less than $\s/2$.

Now we define
\begin{equation}\label{3:V1}
   V_{\e}(z)  =\re\left[\s^2\sum_{Q\in\Qc}\left(\log\left(1-\frac{z}{a_Q}\right)+
   \frac{z}{a_Q}+\frac12\frac{z^2}{a_Q^2}\right)\right].
\end{equation}
It is clear that the series in \eqref{3:V1} converges
uniformly on compacts not containing the points $a_Q$
and thus \eqref{3:V1} defines a harmonic function in
the plane, with these points removed. Due to symmetry,
we can express $ V_{\e}(z)$ via the sum only over the
domains $Q$ belonging to $\Qc_1$, i.e., lying in the
right half-plane,
\begin{equation}\label{3:V1+}
    V_{\e}(z)=\re\left[\s^2\sum_{Q\in\Qc_1}\left(\log\left(1-\frac{z^2}{a_Q^2}\right)+
    \frac{z^2}{a_Q^2}\right)\right].
\end{equation}

The function $V_{\e}(z)$ will be approximated by the
real part of the integral
\begin{equation}\label{3:int2}
     W_\e(z)=\int\limits_{-\e}^\e d\theta \int\limits_1^\infty
     \left(\log\left(1-\frac{z^2}{\t^2}e^{-2i\theta}\right)+\frac{z^2}{\t^2}e^{-2i\theta}\right)\t d\t.
\end{equation}
The behavior of  $W_\e(z)$ is studied in the Appendix.
Let us estimate the difference $V_{\e}(z)-\re W_\e(z)$
for $2\e<|\arg z|<\pi-2\e$, i.e. outside some
sectorial neighborhood of $\Theta_j$ (assuming
$\e<\pi/8$).
\begin{gather}\label{3:V-W}
    V_{\e}(z)-\re W_\e(z)=\\\nonumber
    \sum_{Q\in\Qc_1}
    \re\left[\iint\limits_Q\left(\log\left(1-\frac{z^2}{a_Q^2}\right)+\frac{z^2}{a_Q^2}
    -\log\left(1-\frac{z^2}{w^2}\right)-\frac{z^2}{w^2}\right)d\m(w)\right].
\end{gather}

 To estimate a single term in \eqref{3:V-W}, consider the function
$\b(w)=\log(1-\frac{z^2}{w^2})+\frac{z^2}{w^2}$, for
$2\e<|\arg z|<\pi-2\e$. We have
\begin{gather*}\label{3:V-W:1}
    \iint\limits_Q \b(w)d\m(w)=\\ \b(a_Q)+\b'(a_Q)\iint\limits_Q(w-a_Q)d\m(w)+
    O\left(\iint\limits_Q|\b''(w)|d\m(w)\right).
\end{gather*}
Since $\iint\limits_Q(w-a_Q)d\m(w)=0$, we get the
estimate
\begin{equation}\label{3:V-W:2}
    |V_{\e}(z)-\re W_\e(z)|\le C\sum_{Q\in\Qc_1}\int\limits_Q|\b''(w)|d\m(w).
\end{equation}
Next,
\begin{equation*}
 \b''(w)=\frac{2z^2}{w^2}
 \frac{3w^2-z^2}{(w^2-z^2)^2}+6z^2w^{-4},
\end{equation*}
 therefore,
  $ |\b''(w)|\le C|z|^2|w|^{-4}$,
 so, finally,
\begin{equation}\label{3:V-W:3}
     |V_{\e}(z)-\re W_\e(z)|\le C|z|^2\iint\limits_{|w|\ge1,|\arg w|<\e}
     |w|^{-4}d\m(w)\le \e|z|^2.
\end{equation}
Now we are going to estimate \eqref{3:V-W} in the
sectors around $x_1$-axis, $|\arg z-j\pi|\le 2\e$ for
$j=0$ or $j=1$. Of course, $V_\e$ has logarithmic
singularities at all points $z=\pm a_Q$ and $W_\e$ has
not. We surround each point by a small
disk, $|z\pm a_Q|\le \frac{\s}{4}$
and consider first this difference for $z$ lying
outside all these disks. We cut the angles into three
parts,
\begin{equation*}
    \G_1:|w|\ge2|z|;\; \G_2:|w|\le\frac12|z|;\; \G_3:|w|\in\left(\frac12|z|,2|z|\right).
\end{equation*}
Correspondingly, we denote by $\Qc_{1,j}$ the set of
those domains $Q\in\Qc_1$ for which $a_Q\in\G_j$,
$j=1,2,3$. We suppress the $z$-dependence of these
sets  in notations.

For $Q\in \Qc_{1,1}$, $w\in Q$ we have $5|w|\ge|z-w|,\
|z+w|\ge \frac{|w|}{2}$, therefore the quantity
$w^2(3w^2-z^2)/(w^2-z^2)^2$ is bounded and thus

\begin{equation}\label{3:b1}
|\b''(w)|\le C|z^2w^{-4}|.
\end{equation}
Summing over $Q\in \Qc_{1,1}$, we get
\begin{equation}\label{3:b1.1}
    \sum_{Q\in\Qc_{1,1}}\iint\limits_Q \b''(w)|d\m(w)\le C\iint\limits_{\G_1}|z^2w^{-4}|d\m(w)\le C.
\end{equation}

For $Q\in\Qc_{1,2}$, $w\in Q$, we note that
$5|z|\ge|z-w|,\ |z+w|\ge \frac{|z|}{2}$, therefore the
quantity $w^2(3w^2-z^2)/(w^2-z^2)^2$ is  bounded and
we again arrive at \eqref{3:b1}. Thus
\begin{equation}\label{3:b2.1}
    \sum_{Q\in\Qc_{1,2}}\iint\limits_Q |\b''(w)|d\m(w)\le C\iint\limits_{\G_2}|z^2w^{-4}|d\m(w)\le C\e|z^2|.
\end{equation}

The region $\G_3$ requires a harder work. In any
$Q\in\Qc_{1,3}$, we write
\begin{gather}\label{3:b3}\nonumber
        \re\left(\s^2\log\left(1-\frac{z^2}{a_Q^2}\right)+\frac{z^2}{a_Q^2}\right)-
        \re\iint\limits_Q\left(\log
        \left(1-\frac{z^2}{w^2}\right)+\frac{z^2}{w^2}\right)d\m(w)\\
=\iint\limits_Q\re\left(
\log(1-\frac{z^2}{a_Q^2})-\log
\left(1-\frac{z^2}{w^2}\right)\right)d\m(w)+\\
\nonumber\iint\limits_Q\re\left(\frac{z^2}{a_Q^2}-\frac{z^2}{w^2}\right)d\m(w).
\end{gather}
Consider the second term in \eqref{3:b3}. We have
$\frac{z^2}{a_Q^2}-\frac{z^2}{w^2}=z^2(z+w)w^{-2}a_Q^{-2}$
$(z-w)$. The quantities $|z|,|w|,|a_Q|$ are of the
same order, while $|z-w|\le 2\s\e^{-1/2}$. (Of course,
$|z-w|\le\s\sqrt{2}$ if $Q$ is a unit square, but if
$Q$ is a triangle or a trapezium, only the bound by $
2\s\e^{-1/2}$ is guaranteed.) Therefore, the second
term in \eqref{3:b3} is majorized by
$C\s\e^{-\frac12}|z^{-1}|$, and since the quantity of
domains in $\Qc_{1,3}$ is of order $\s^{-2}\e|z|^2$,
we obtain the estimate
\begin{equation}\label{3:b3.sum}
    \sum_{Q\in\Qc_{1,3}}\iint\limits_Q \re\left(\frac{z^2}{a_Q^2}-
    \frac{z^2}{w^2}\right)d\m(w)\le C\s^{-1}\sqrt{\e}|z|.
\end{equation}
Next we  estimate the first term in \eqref{3:b3}. We
transform the integrand as
\begin{equation}\label{3:b3.log}
\log\!\left(\!1-\frac{z^2}{a_Q^2}\!\right)\!+\!\frac{z^2}{a_Q^2}\!-\!\log\left(\!1\!-\!\frac{z^2}{w^2}\!\right)\!-\!\frac{z^2}{w^2}\!=
\!\log\frac{a_Q-z}{w-z}\!+\!
\log\frac{a_Q+z}{w+z}\!+\!2\log\frac{w}{a_Q}.
\end{equation}
In the second term in \eqref{3:b3.log} we write
\begin{equation}\label{b3.log.2}
    \left|\log\frac{a_Q+z}{w+z}\right|=\left|\log\left(1+\frac{a_Q-w}{z+w}\right)\right|\le
     \frac{C}{\sqrt{\e}|z|}.
\end{equation}
Similarly, the third term in \eqref{3:b3.log} is
estimated as
\begin{equation}\label{b3.log.3}
\left|\log\frac{w}{a_Q}\right|=\left|\log\left(1+\frac{w-a_Q}{a_Q}\right)\right|\le
\frac{C}{\sqrt{\e}|z|}.
\end{equation}
Now we pass to the first term in \eqref{3:b3.log}. We
split the sum into two: the sum over such $Q$ that
$|a_Q-z|\le10\s/\sqrt{\e}$ and the sum over the
remaining $Q.$ Consider the first, finite sum (recall
that $|z-a_Q|>\frac{\s}{4}$):
\begin{equation}\label{3:b3.log.1a}
    \sum_{Q\in\Qc_{1,3},
    |a_Q-z|\le10\s/\sqrt{\e}}\iint\limits_Q\left|\log\left(
    \frac{a_Q-z}{w-z}\right)\right|d\m(w)\le
    C|\log\s|\e^{-1}.
\end{equation}
For the second sum  we have
$|w-z|\ge\frac{5\s}{\sqrt{\e}}$ therefore for $w\in Q$
\begin{equation*}
\left|\log\left(\frac{a_Q-z}{w-z}\right)\right|=\left|\log\left(1+
\frac{a_Q-w}{w-z}\right)\right|\le C|w-z|^{-1},
\end{equation*}
and thus
\begin{equation}\label{3:b3.log.1b}\begin{split}
\sum_{Q\in\Qc_{1,3},
\!\!|a_Q-z|\ge\frac{10\s}{\sqrt{\e}}}\iint\limits_Q\left|\log\left|\frac{a_Q-z}{w-z}\right|\right|d\m(w)\le
C\!\!\!\!\!\!\iint\limits_{|w|\in(|z|/2,2|z|),|w-z|\ge\frac{5\s}{\sqrt{\e}}
}\!\!\!\!\frac{d\m(w)}{|w-z|}\le C\s|z|.\end{split}
\end{equation}

Summing the estimates \eqref{3:V-W:2},
\eqref{3:V-W:3}, \eqref{3:b1.1}, \eqref{3:b2.1},
\eqref{3:b3.sum},  \eqref{b3.log.2}, \eqref{b3.log.3},
\eqref{3:b3.log.1a}, \eqref{3:b3.log.1b}, we obtain
the following inequality.
\begin{proposition}\label{3:Prop}
For a given $\e$ and functions $V_\e(z)$ and $W_\e(z)$
defined as in \eqref{3:V1} and \eqref{3:int2},
\begin{equation*}
    |V_\e(z)-\re W_\e(z)|\le C \e|z|^2+ C'|\log\s|\e^{-1}+C\s|z|,
\end{equation*}
as $|z|\to\infty$ and $z$ avoids
$\frac{\s}{4}$-neighborhoods of the points $a_Q$. In
particular, using the asymptotics \eqref{5:int3}  for
$W_\e$, we have
\begin{equation}\label{3:prop.eq1}
    V_\e(z)=\frac12|z|^2\log|z|\sin(2\e)\cos(2\p)+\e O(|z|^2)+O(\log|\s|)+C\s|z|,
\end{equation}
if $z$ tends to infinity along the line
$z=|z|e^{i\p}$.
\end{proposition}
It remains to estimate the difference in question for
$z$ in $\s/4$ neighborhood of the point $a_{Q_0}$ for
some $Q_0\in\Qc$. Note that by our construction, there
can be only one such point $a_{Q_0}$. Here we can
simply separate the term corresponding to $Q=Q_0$ in
the sum \eqref{3:V-W}. For the sum of remaining terms
the inequality we just obtained holds. This gives us
the following estimate.
\begin{proposition}\label{3:prop1}
For a given $\e$ and functions $V_e(z)$ and $W_\e(z)$
defined as in \eqref{3:V1} and \eqref{3:int2},
\begin{equation}\label{3:prop.eqq}
    \left|V_\e(z)-\re W_\e(z)-\s^2\log\left(1-\frac{z}{a_{Q_0}}\right)\right|\le C
\e|z|^2+ C'|\log\s|\e^{-1}+C\s|z|,
\end{equation}
as $|z|\to\infty$ and $z$ lies in the
$\s$-neighborhood of the point $a_{Q_0}$.
\end{proposition}

The estimates \eqref{3:prop.eq1}, \eqref{3:prop.eqq}
lead to the following inequalities for the exponent of
$V_\e(z)$.
\begin{proposition}\label{3:prop3} For a positive
$\k$, we have
\begin{equation}\label{3:expV1}
    |\exp(\k V_\e(z))|\le \exp(\k\frac12|z|^2\log|z|\sin(2\e)\cos(2\p)+\e
    O(\k|z|^2)+O(\log|\s|)),
\end{equation}
as  $z$ tends to infinity along the line
$z=|z|e^{i\p}$.\end{proposition}
\begin{proof}
If $z$ tends to infinity along the line $z=|z|e^{i\p}$
avoiding the $\s/4$-neighborhoods of the points $a_Q$,
\eqref{3:expV1} follows from \eqref{3:prop.eq1}. For
$z$ in these neighborhoods, we apply
\eqref{3:prop.eqq} and use that
$|\exp(\log\left(1-\frac{z}{a_{Q_0}}\right))|$ is
bounded.
\end{proof}
\subsection{Construction of entire functions} We return to our
initial problem. We recall that $b_2=1$ and that
$|b_1|$ is sufficiently small (how small will be
determined later). We set
   $ \kappa=\frac{c_0\sin(\a)}{\pi \sin(2\e)};\
   c_0=1-b_1.$

We chose $\s=\k^{-\frac12}$ and
   define
function $\Phi(z)$ as   the Weierstrass product:
\begin{equation*}
    \Phi(z)=\prod_{Q\in\Qc_+}\left[\left(1-\frac{z^2}{a_Q^2}\right)
    \exp\left(\frac{z^2}{a_Q^2}\right)\right].
\end{equation*}
We also denote
\begin{equation*}
    \Phi_\a(z)=\F\left(ze^{-i\frac{\a+\pi}{2}}\right).
\end{equation*}
By our choice of $\s$, this function is related to the
function $V_\e$ considered in Section~\ref{3}:
\begin{equation*}
    \log\left|\F_\a(z)\right|^2=\s^{-2} V_\e\left(ze^{-i\frac{\a+\pi}{2}}\right)=\k V_\e\left(ze^{-i\frac{\a+\pi}{2}}\right).
    \end{equation*}
Therefore, by \eqref{3:prop.eq1}, with $z=re^{i\p}$
\begin{equation}\label{4:FiVfor}\begin{split}
 \log\left|\F_\a(z)\right|^2\le\frac12 |z|^2\log|z|\sin(2\e)\k
  \cos(2(\p-\frac{\a+\pi}{2}))\\
 +\k(\e+1)O(|z|^2)=
 -\frac{c\sin{\a}}{2\pi}|z|^2\log|z|\cos(2(\p-\a/2))\\+\sin\a
 (\e+1)/\sin(2\e)O(|z|^2).\end{split}\end{equation}
With $\e$ chosen as $\a^{1/2}$ (thus
$\s\sim\a^{-\frac12}$), we have
\begin{equation*}
\sin\a
 (\e+1)/\sin(2\e)\le C\a^{1/2}.
\end{equation*}

For a function $G(z)$, to be specified later, we
consider the integral
 \begin{equation}\label{integral1}
    I(G)=\iint\limits_{\C}e^{2F(z)}|\F_\e(z)|^2 |G(z)|^2
    d\m(z)=\iint\limits_{\C}e^{2F(z)-2\log|\F_\e(z)|}|G(z)|^2
    d\m(z).
\end{equation}
From  the estimates for $F$ in \eqref{2:Final} and for
$\log\F$ in \eqref{4:FiVfor}, we see that the terms
with  $|z|^2\log|z|$  in  the exponent in
\eqref{integral1} cancel and therefore
\begin{equation}\label{integral2}\begin{split}
    I(G)\le
    C\iint\limits_{\O_1}e^{|b_1|x_2^2+c_0\a^{1/2}|z|^2}|G(z)|^2d\m(z)\\
    +C\iint\limits_{\O_2}e^{-x_2^2+c_0\a^{1/2}|z|^2}|G(z)|^2d\m(z)
    =I_1(G)+I_2(G).
\end{split}\end{equation}
Now we choose the function $G(z)$. We take it in the
form
\begin{equation*}
    G(z)=\exp(-1/4\left(ze^{-i\frac{\a}{2}}\right)^2)P(z),
\end{equation*}
where $P(z)$ is an arbitrary polynomial. Then
\eqref{integral2} implies
\begin{equation*}
    \begin{split}
    I_1(G)\le
    C\iint_{\O_1}e^{|b_1|\sin^2\a|z|^2+c_0\a^{1/2}|z^2|-\frac12|z|^2\cos2(\p-\a/2)}d\m(z)\\
    \le\iint_{\O_1}e^{(|b_1|\a^2+c_0\a^{1/2})|z^2|-\frac12
    |z^2|\cos(\a)}|P(z)|^2d\m(z)<\infty,
    \end{split}
\end{equation*}
as soon as $\a, |b_1|\a$ are sufficiently small.
Further on, we split the integral $I_2(G)$ as
$I_2(G)'+I_2(G)''$, so that $I_2(G)'$ involves
integration over the region in $\O_2$ where $|\arg
z|<\a^{1/5}$ or  $|\arg z-\pi|<\a^{1/10}$, and
$I_2(G)''$ involves the integration over the rest of
$\O_2$, i.e., the region where $|\arg z|>\a^{1/10}$and
$|\arg z-\pi|>\a^{1/5}$. This gives us
\begin{equation}\label{integral4}\begin{split}
I_2(G)=I_2(G)'+I_2(G)''\le\\
\iint_{\O_2}e^{-\frac12\cos(\frac32
\a^{1/5})|z|^2+c_0\a^{1/2}|z|^2}|P(z)|^2
d\m(z)+\\\iint_{\O_2}e^{-\sin^2\a^{1/5}|z|^2+c_0\a^{1/2}|z|^2}|P(z)|^2
d\m(z).\end{split} \end{equation} Both integrals in
\eqref{integral4} converge,  again, as soon as $\a$ is
small enough.

Thus any entire analytical  function $u(z)$ of the
form
\begin{equation*}u(z)=\F_\a(z)\exp(-1/4\left(ze^{-i\frac{\a}{2}}\right)^2)P(z)\end{equation*}
belongs to $L_2(\C)$ with weight $2F(z)$ and therefore
the dimension of the corresponding subspace is
infinite.
\section{Nonexistence of zero modes}\label{nonex}
In this Section we prove the following theorem about
the non-existence of zero modes.
\begin{theorem}\label{thm:nonex}
 Suppose that  the angle $\a$ is
sufficiently close to $\pi$ and $|b_1|<\frac12$. Then
the space of analytical functions $u$ satisfying
$u\exp(\pm F)\in L_2$ consists only of the zero
function.\end{theorem}
\subsection{A half-plane}\label{half-plane} We consider the case of $\a=\pi$ first.
 So, let us have $B(x)=2b_1<0$ in the
half-plane $\C^+=x_2 >0$ and $B(x)=2b_2=2$ in the
half-plane $\C^-=x_2<0$. The potential, the solution
of the equation $\D F=B$ can be taken in the form
\begin{equation*}
    F(z)=b_1 x_2^2,x_2>0;
   F(z)=x_2^2,x_2<0.
\end{equation*}
We will show  that no nontrivial entire analytical
function $u(z)$ can belong to $L_2$ with weight
$e^{F(z)}$ or $e^{-F(z)}$. Actually, a more general
statement is correct.
\begin{proposition}\label{nonex.half}
Let $h(s), s\ge0$ be a positive function,
$h(s)\to\infty$ as $s\to\infty$. Then the set of
functions $u(z)$, analytical in the half-plane $x_2>0$
and continuous up to the boundary such that
\begin{equation}\label{4:finiteInt:eq}
    \iint_{\C^+}e^{x_2h(x_2)}|u(z)|^2 d\l(z)<\infty
\end{equation}
consists only of a zero function.
\end{proposition}

It is clear that the absence of nontrivial  entire
functions in the case we started with   follows from
Proposition \ref{nonex.half} applied separately to
half-planes $\C^+$ and $\C^-$. Moreover, Proposition
\ref{nonex.half} will be the destination point for
other configurations of $B$ to be considered: having
obtained a lower estimate for the function $F$ in some
domain, we make a conformal mapping of this domain
onto the upper half-plane, where the Proposition can
be applied.

\begin{proof} It follows from the condition \eqref{4:finiteInt:eq}
that for any fixed $y_0>0$, the function
$u_{y_0}(x_1)=u(x_1+iy_0)$ belongs to $L_2(\R^1)$ as a
function of $x_1$, moreover, $u(x_1+ix_2)$ tends to
$0$ as $x_2\to\infty$, uniformly in $x_1\in\R^1$. Thus
$u(x_1+ix_2)$ is a bounded harmonic function in a
half-plane $\C_+^{y_0}={x_2\ge y_0}$ with boundary
values in $L_2(\R^1)$, Such function can be expressed
by means of the Fourier transform:
\begin{equation*}
    u(x_1+ix_2)=\Fc^{-1}_{\xi\to x_1}
    e^{-(x_2-y_0)|\xi|^2}\widehat{u_{y_0}}(|\xi|).
\end{equation*}
Therefore,
\begin{equation*}
    \iint_{x_2>y_0}e^{x_2h(x_2)}|u(x_1+ix_2)|^2d\m(z)=
    \iint_{x_2>y_0}e^{x_2h(x_2)-2(x_2-y_0)|\xi|}|\widehat{u_{y_0}}(\xi)|^2d\xi d
    x_2,
\end{equation*}
and since $h(s)\to\infty$ as $s\to\infty$, the
integral diverges unless $u_{y_0}\equiv 0$ for any
$y_0>0$.
\end{proof}

The case of a sector configuration of $B$ will be
reduced to Proposition \ref{nonex.half} by means of a
conformal mapping. The following subsection will be
devoted to the proof that such special mapping  is, in
fact univalent.

\subsection{Univalentness property}
\begin{proposition}\label{4:unival}
 Denote for $R>0$ by $\C^{+}_R$ the
upper half-plane $\C^+$ with the the disk $|z|\le R$
removed. Then for any $A\ne0$ there exists a number
$R_A>1$ such that the function
\begin{equation*}
    \z(\o)=\o(\log \o -\frac{\pi}{2}i)^A
\end{equation*}
is analytical and univalent in $\C^+_{R_A}$ and maps
it onto a set $\O_A\subset\C$. The boundaries of
$\O_A$ are described by
\begin{equation}\label{4:domain}
   \begin{split}
    \vk = -\frac{\pi A}{2\log\r}+O(\frac{\log\log\r}{\log^2\r}),\ \z=\r e^{i\vk},\vk \ {\rm near }\ 0;\\
      \vk=  \pi +\frac{\pi A}{2\log\r}+O(\frac{\log\log\r}{\log^2\r}), \
\z=\r e^{i\vk}\vk \ {\rm near}\ \pi.
\end{split}
\end{equation}
In other words, the function $\z(\o)$ maps conformally
the upper half-plane, with a disk cut away, onto a
slightly, logarithmically, deformed half-plane, with a
compact set cut away.
\end{proposition}
\begin{proof} The fact that the function $\xi$ is
analytical in the domain $\C^+_{R_A}$ and the
asymptotic expressions \eqref{4:domain} for the
mapping of the boundaries follow directly from the
definition of the function. What, actually, requires
being checked is that the function is univalent for
$R=R_A$ sufficiently large. We start with an
intermediate mapping onto a strip. Set
\begin{equation}\label{4:intermed}
    z=z(\o)=\log\o -\frac\pi2 i.
\end{equation}
The mapping \eqref{4:intermed} transforms $\C^+$ onto
the strip $\{z=x+iy: |y|<\frac\pi2\}$ and the domain
$\C_R$ onto the half-strip  $\Pi_\varsigma=\{z=x+iy:
|y|<\frac\pi2, x>\varsigma\}$, $\varsigma=\log R$.
Since $\o=ie^z$, it is sufficient to check that the
function $\z(z)=e^z z^A$ is univalent in
$\Pi_\varsigma$ for $\varsigma$ large enough.

We choose $\varsigma$ so that $\varsigma>2\pi$ and
moreover
\begin{equation}\label{4:strip1}
    |\arg (z^A)|<\frac{\pi}{40},\
    |\arg(1+Az^{-1})|<\frac{\pi}{40},\ z\in \Pi_\varsigma.
\end{equation}
We show first that the function $\z(z)$ is univalent
in any substrip
\begin{equation*}
D=\{z=x+iy: \ x>\varsigma, |y-y_D|<0.4\pi\},
\end{equation*}
such that $\varsigma$ satisfies \eqref{4:strip1} and
$D\subset \Pi_\varsigma$. Let $ z_D=\varsigma+1+iy_D,$
\begin{equation}\label{4:strip3}
\n_D=\arg(\z'(z_D)) =y_D+\arg (z_D^A)+\arg(1+A/z_D).
\end{equation}
By \eqref{4:strip1}, for any $z=x+iy\in D$, we have
\begin{gather}\label{4:strip4}
    |\arg \z'(z)-\arg \z'(z_D)| \\ \nonumber
    \le |y-y_D|+|\arg z^A|+|\arg(1+A/z)|+|\arg
    z_D^A|+|\arg(1+A/z_D)|\le \\\nonumber
     0.4\pi+4\cdot\frac{\pi}{40}=\frac{\pi}{2}.
\end{gather}
Now let $z_1,z_2$ be some points in $D$,
$\frac{z_2-z_1}{|z_2-z_1|}=e^{i\h}$. Then we have
\begin{equation*}\begin{split}
    \z(z_2)-\z(z_1)=\int\limits_{[z_1,z_2]}\z'(t)dt=e^{i\h}\int\limits_0^{|z_2-z_1|}\z'(z_1+\t
    e^{i\h})d\t\\=e^{i(\h+\n_D)}\int\limits_0^{|z_2-z_1|}e^{-i\n_D}\z'(z_1+\t
    e^{i\h})d\t,
\end{split}
\end{equation*}
and therefore, by \eqref{4:strip3}, \eqref{4:strip4}
\begin{gather*}
    |\z(z_1)-\z(z_2)|=\left|\int\limits_0^{|z_2-z_1|}e^{-i\n_D}\z'(z_1+\t
    e^{i\h})d\t\right|\\\nonumber \ge\left|\int\limits_0^{|z_2-z_1|}\re\left(e^{-i\n_D}\z'(z_1+\t
    e^{i\h})\right)d\t\right|>0.
\end{gather*}
Now we consider the whole half-strip $\Pi_\varsigma$.
Let us  take arbitrary points $z_1,
z_2\in\Pi_\varsigma$. If $|\im(z_1-z_2)|<0.8\pi$ then
$z_1,z_2$ lie in some half-strip $D$ of the type just
considered and thus $\z(z_1)=\z(z_2)$ is impossible.
On the other hand, if $\im(z_2-z_1)\ge0.8\pi$, we have
\begin{equation*}
\arg\z(z_2)-\arg\z(z_1)=\im(z_2-z_1)+(\arg(z_2^A)-\arg(z_2^A))>0.8\pi-\frac{2\pi}{40}>0,
\end{equation*}
and again $\z(z_1)=\z(z_2)$ is impossible.
\end{proof}

\subsection{Estimates for a large angle}\label{LargeAngle-II}
We consider the case when in the setting of Section
\ref{general} the  angle $\a$ is  close to $\pi$.
 We are going  to show  that if $\theta=\pi-\a$ is
small enough then there are no nontrivial entire
functions $u(z)$ such that $u\exp(F)$ or $u\exp(-F)$
belong  to $L_2(\C)$. To do this, we consider two
adjoining sectors $\X_\pm$ with slightly curved
boundaries, where the functions $\pm F$ are positive.
By performing the conformal mapping of $\X_\pm$ onto
the upper half-plane, we arrive at the situation
described in Subsection \ref{half-plane}. Recall that
in polar coordinates $r,\p$ the function $F(z)$ has
the form
\begin{equation*}
    F(re^{i\p})=\f(z)-\frac{c_0\sin{\theta}}{2\pi}r^2\log
    r\cos(2(\p-\a/2))+\frac{c_0\sin{\theta}}{2\pi}\p\sin(2(\p-\frac{\a}{2}))r^2.
\end{equation*}
Recall that the polar angle $\p$ lies in
$[\a,\a+2\pi)$ in this representation, $c_0=1+|b_1|$.

We suppose that $\a$ is close to $\pi$ and  consider
two sectors in the complex plane,
$S=\{z:\p\in(\frac{\a}{2}+\frac74\pi,\frac{\a}{2}+\frac94\pi)
\}$ and
$T=\{z:\p\in(\frac{\a}{2}+\frac94\pi,\a+2\pi)\cap[\a,\frac{\a}{2}+\frac34\pi)\}=T_1\cup
T_2$. In these sectors the log-quadratic part of $F$
is negative, resp. positive.

Consider the sector $S$. If $z=re^{i\p}\in S$, we have
$z\in\O_1$, $\f(z)=b_1r^2\sin^2\p$, and therefore
\begin{gather}\label{4A:0}
-F(z)=\left( |b_1|\sin^{2}\p
-\frac{c_0\sin{\theta}}{2\pi}\p\sin(2(\p-\frac{\a}{2}))\right)r^2\\
\nonumber
+\frac{c_0\sin{\theta}}{2\pi}\cos(2(\p-\frac{\a}{2}))r^2\log
r.
\end{gather}
It follows from \eqref{4A:0}that $-F\ge C r^2\log r$
for sufficiently large $r$ in $S$ with arbitrarily
small  sectors near the boundary of $S$ removed. To
estimate $F$ near  the boundary of $S$,  we chose
$\theta$ so small that
\begin{equation}\label{4A:1}
\frac{c_0\sin{\theta}}{2\pi}(\frac94\pi+\frac{\a}{2})\le
\frac12 \sin^2(\frac94\pi+\frac{\a}{2})|b_1|.
\end{equation}
Then, by \eqref{4A:1}, for $z=re^{i\p}$, $\p$ close to
$\frac94 \pi+\frac\a2$, we have
\begin{equation}\label{4A:2}
\left(
|b_1|\sin^{2}\p-\frac{c_0\sin{\theta}}{2\pi}\p\sin(2(\p-\frac{\a}{2}))\right)r^2\ge
\frac{|b_1|}2 r^2\ge Cr^2.
\end{equation}
for some positive constant $C$. For $\p$ close to
$\frac74 \pi+\frac\a2$ we note that
$\sin(2(\p-\frac{\a}{2}))$ is negative,
therefore\begin{equation}\label{4A:3} \left(
|b_1|\sin^{2}\p-\frac{c_0\sin{\theta}}{2\pi}\p\sin(2(\p-\frac{\a}{2}))\right)r^2\ge
Cr^2
\end{equation}

 Now, it follows from \eqref{4A:2}, \eqref{4A:3} that  for some constant $A>0$, small
enough, the quadratic term in \eqref{4A:0} majorates
the log-quadratic term in the domains $S_1(A)=\{z=r
e^{i\p}, |\p-(\frac74\pi+\frac{\a}{2})| \le
\frac{A}{\log r}\}$ and $S_2(A)=\{z=r e^{i\p},
|\p-(\frac94\pi+\frac{\a}{2})| \le \frac{A}{\log
r}\}$:
$$\left|\frac{c_0\sin{\theta}}{2\pi}\cos(2(\p-\frac{\a}{2}))r^2\log
r\right|\le \frac{C}{2}r^2,\ z\in S_1(A)\cup S_2(A).$$
Therefore, in the domain $S'=S\cup S_1(A)\cup S_2(A)$,
$S'$ slightly larger than the quarter-plane, we have
\begin{equation*}
    -F(z)\ge \frac{C}{2}|z|^2.
\end{equation*}

Now suppose that for some nontrivial analytical
function $u(z)$ we have
\begin{equation}\label{4A:contr}
    \iint_{S'} |u(z)|^2 e^{-2F(z)} d\m(z)<\infty
\end{equation}
 In order to obtain a contradiction, we
make a conformal mapping  of the domain  $S'$ onto a
domain covering the upper halve-plane. We will do it
in two steps.

 Let $\z=\r e^{i\vk}=(ze^{i\frac74\pi})^2.$
Under this mapping, the domain $S'$ is transformed
conformally  onto
 \begin{equation*}
 \tilde S=\{\z:\ \r> r_0^2, \
-\d(\r)<\arg\vk<\pi+\d(\r) ,\end{equation*} where
 \begin{equation*}
 \d(\r)\sim
\frac{2A}{\log \r}.\end{equation*}

Next we make a change of variables in the integral in
\eqref{4A:contr} by setting
${v(\z)}=\frac{u(\sqrt{\z})}{\sqrt{\z}}$. We obtain
\begin{equation}\label{4A:assumption1}
    \iint\limits_{\tilde S}|v(\z)|^2e^{|\z|}d\m(\z)<\infty.
\end{equation}
As it follows from  our construction, the domain
$\tilde S$ is slightly, logarithmically, larger than
the upper half-plane with a disk removed. Now we map
conformally this set onto the upper half-plane with a
compact set removed. It is more convenient to do it by
considering the inverse mapping.

 We set
$\z=\z(\o)=\o(\log(\o-\frac{\pi i}{2}))^A$. If $A$ is
small enough and $a$ is large enough, the image under
this mapping of the upper half-plane with the disk
$|\o|<a$ removed, lies in $\tilde S$. By Proposition
\ref{4:unival}, the mapping $\z(\o)$ is univalent in
this set, so the inverse, $\o=\o(\z)$ exists, maps the
image of $\z(\o)$ onto the $\C_a$ and its asymptotics
as $|\z|\to\infty$ can be easily found. We change
variables in the integral in \eqref{4A:assumption1}
which gives
\begin{equation}\label{4:assumption4}
    \iint_{\C^+_{a}}|v(\z(\o))|^2|\z'(\o)|
    e^{\frac{C}{4}|\o||\log|\o||^A}d\m(\o)<\infty.
\end{equation}
Since $|\z'(\o)|$ behaves  logarithmically at
infinity, we can apply  Proposition \ref{nonex.half}
with $h(x_2)=(\log(|x_2|+1))^A-C$. Thus,
\eqref{4:assumption4} can hold only for  $v\equiv 0$,
or $u\equiv 0$.

Now we consider the sector $T$. From the expression
for $F(z)$ in \eqref{2:Final} we obtain for  small
$\theta$
\begin{equation}\label{4A:S1}
    F(re^{i(\frac{\a}2+\frac34\pi})=(\frac12+O(\theta))r^2\ge\frac14r^2.
\end{equation}

We consider now an auxiliary entire analytical
function $H(z)=-\frac{i}6 e^{-\a}z^2$, so that $\re
H(z)=-\frac16 r^2\sin2(\p-\frac\a2-\frac\pi2)$. Then,
by \eqref{4A:S1} and \eqref{4A:2} for the function
$\tilde F(z)=F(z)-\re H(z)$ we have
\begin{equation*}
\tilde F(z)\ge \frac1{12}r^2 \ {\rm for}\
\p=\frac94\pi+\frac\a2,\ \tilde F(z)\ge\frac1{24}r^2 \
{\rm for}\ \p=\frac34\pi+\frac\a2.
\end{equation*}
So,  the log-quadratic term in $\tilde F$ is positive
in the sector $T$ and the quadratic term in $\tilde F$
is positive on the boundaries of $T$. Therefore,
similar to the above consideration in the sector $S$,
the function $\tilde F$ admits a quadratic lower
estimate in a domain slightly larger than $S$:
\begin{equation*}
\tilde F(z)\ge C|z|^2,
\p\in[\frac\a2,\frac34\pi+\frac{A}{\log
r}]\cup[\frac94\pi+\frac\a2-\frac{A}{\log
r},\frac\a2+2\pi].
\end{equation*}
If an entire analytical function $u$ satisfies
$\iint_{\C}e^{2F}|u|^2\d\m(z)<\infty$ then the
function
 $\tilde u=u\exp(H)$ satisfies $\iint_{\C}e^{2\tilde F}|\tilde
 u|^2\d\m(z)<\infty$.
 It remains to repeat the  reasoning with the conformal
 mapping used for the sector $S$ to show that the
 function $\tilde{u}$ and therefore $u$ must
 necessarily be zero.

 This concludes the proof of Theorem \ref{thm:nonex}.

\section{The Non-resonance Case}
As it can be observed from the calculations above, the
main trouble in the study of the 'sector'
configuration is created by power-logarithmic behavior
of the function $F$.  In the non-resonance case such
terms are not present in $F$, therefore the analysis
of zero modes is considerably easier.
\begin{theorem} Let $B(z)=B(re^{i\p})$ be radial homogeneous
of degree $s\in(-2,0]$ and
$\b_0=(2\pi)^{-1}\int_0^{2\pi}B(e^{i\p})d\p\ne 0$. For
$s=0,-1$ we suppose that
$\int_0^{2\pi}B(e^{i\p})e^{i(s+2)\p}d\p=0$. Then, if
$\int_0^{2\pi}|B(e^{i\p})-\b_0|^2d\p$ is sufficiently
small then the space of zero modes is
infinite-dimensional.\end{theorem}
\begin{proof} Suppose that $\b_0>0$ and set
 $\tilde B(x)=B(x)-\b_0 r^s$. Then the solution
$F$ of the Poisson equation $\D F=B$ can be
represented as $F=\F+\tilde{F}$, where
$\F(x)=\b_0(s+2)^{-2}|r|^{s+2}$ and $\tilde
F=\varphi(\p)r^{s+2}$ with $\varphi(\p)$ being a
solution of $\varphi''(\p)+(s+2)^2\varphi(\p)=\tilde
B(\p)$. Such solution exists (for $s=0$ or $s=-1$ we
use the orthogonality condition and require also that
$\varphi(\p)$ is again orthogonal to $e^{(s+2)\p}$ ).
Moreover, the solution $\varphi(\p)$, by ellipticity,
belongs to the Sobolev space $H^2$ on the circle
$\mathbf{S}^1$ with estimate
$||\varphi||_{H^2(\mathbf{S}^1)}\le C ||\tilde
B||_{L_2(\mathbf{S}^1)}$. Thus, if the latter norm is
small enough, then, by the embedding theorem,
$||\varphi||_{C(\mathbf{S}^1)}$ ia also small and can
be made smaller than $|\b_0|$. In this case, it turns
out that $F(x)\ge C|x|^{s+2}$ and therefore for any
polynomial $p(z)$ the integral
$\iint_{\C}|p(z)|^2e^{-2F}d\m(z)$ converges.
\end{proof}

We explain here the role of the orthogonality
condition in the cases $s=0$ and $s=-1$. If it is
violated, the solution of the Poisson equation
 contains necessarily a log-power term, similar to the
one considered in Sections 3,4, and therefore will not
be sign-definite. The same complication arises in the
case $s=-2$.

Finally, we present a construction showing that in the
nonresonance case the absence of zero modes can also
occur.

\begin{example}
Let $s\in(-1,0]$. We construct the  function
$F(z)=F(re^{ip})=f(\p)r^{s+2}$ in the following way.
For some $\e>0$, $\e<\frac14(1+s)$, we consider two
disjoint  arcs $I_+,I_-$ in $\mathbf{S}^1$ having
length $\pi(s+2-\e)^{-1}<\frac\pi2$. We set
$f(\p)=\b_+>0$ on $I_+$, $f(\p)=\b_-<0$ on $I_-$ with
some constants $\b_{\pm}$ and define $f(\p)$ in an
arbitrary way on the complement of $I_{\pm}$, to
obtain a smooth function on the circle. Denote by
$S_{\pm}$ the sectors in $\C$ defined by the arcs
$I_{\pm}$. Supposing that there exists an analytical
function $u(z)$ satisfying
$\iint_{S_+}e^{2F}|u(z)|^2d\m(z)<\infty$, we make a
conformal mapping $z(\z)$ of the upper half-plane
$\C_+$ onto the  sector $S_+$,
$z=z_0\z^{(s+2-\e)^{-1}}$, $|z_0|=1$. Then the
integral transforms to
\begin{equation}\label{nonres:1}
    (s+2-\e)^{-1}\iint_{\C_+}e^{2F(z_0\z^{(s+2-\e)^{-1})}}|u(z_0\z^{(s+2-\e)^{-1}})|^2
    |\z|^{(s+2-\e)^{-1}-1}d\m{\z}.
\end{equation}
For the exponent $2F(z_0\z^{(s+2-\e)^{-1}})$ we have
the lower estimate by $|z|^{\frac{s+2}{s+2-\e}}$, and
by Proposition \ref{nonex.half} the  function $u$
should be zero. In a similar way, the integral
$\iint_{S_-}e^{-2F}|u(z)|^2d\m(z)$ cannot be finite
unless $u\equiv 0$.
\end{example}

\appendix \section{Some integrals} We present here the
calculation of the integral in \eqref{3:int2}. we show
first that
\begin{equation}\label{4:int1}
    v(z)=\int\limits_1^\infty\left(t\log(1-\frac{z^2}{t^2})+\frac{z^2}{t^2}\right)dt=
    \frac{z^2-1}{2}\log(1-z^2)-\frac{z^2}{2}.
\end{equation}

It is clear that $v(0)=0$. We find the derivative of
$v(z)$. For $|z|<1$ it is legal to
 differentiate under the
integral sign, therefore
\begin{equation*}
    v'(z)=-z\int\limits_1^\infty(\frac{1}{t-z}+\frac{1}{t+z}-\frac{2}{t})dt=z\log(1-z^2).
\end{equation*}
The derivative of the right-hand side in
\eqref{4:int1} gives the same expression.

For $|z|\ge1, \im z\ne0$, we can continue analytically
the expression in  \eqref{4:int1} separately to the
upper and lower half-planes, and the corresponding
branches of the logarithm should be used.

 Next we consider the integral in  \eqref{3:int2}. We represent it, using \eqref{4:int1} as
 \begin{equation}\label{4:int2}\begin{split}
    W_\e(z)=\int\limits_{-\e}^\e d\vartheta\int\limits_1^\infty
     \left(\log(1-\frac{z^2}{\t^2}e^{-2i\vartheta})
    +\frac{z^2}{\t^2}e^{-2i\vartheta}\right)\t d\t=\\
    \frac12 \int\limits_{-\e}^\e(ze^{i\theta})^2\log(1-(ze^{i\theta})^2)d\theta-\frac12
    \int\limits_{-\e}^\e(ze^{i\vartheta})^2d\vartheta-\frac12
    \int\limits_{-\e}^\e\log(1-(ze^{i\vartheta})^2)d\vartheta.
 \end{split}\end{equation}

 We estimate the  $\vartheta$ integral  for small $\e$, obtaining
 \begin{equation}\label{5:int3}
W_\e(z)=\frac12|z|^2\log|z|\sin(2\e)\cos(2\phi)+\e
O(|z|^2)
\end{equation}
as $z$ tends to infinity along the line
$z=|z|e^{i\phi}$.

\end{document}